\documentclass[11pt]{article}

\usepackage[a4paper,margin=1in]{geometry}

\usepackage{enumitem}
\usepackage{amssymb}
\usepackage{amsmath}
\usepackage{tikz}
\usepackage{tikz-cd}
\usetikzlibrary{arrows.meta, positioning}

\usepackage{setspace}
\usepackage{xurl}
\usepackage{hyperref}

\hypersetup{
  colorlinks=true,
  linkcolor=blue,
  citecolor=blue,
  urlcolor=blue,
  breaklinks=true
}
\usepackage{cleveref}
\crefname{equation}{}{}
\Crefname{equation}{}{}

\usepackage{amsthm}

\theoremstyle{plain}
\newtheorem{theorem}{Theorem}[section]

\theoremstyle{definition}

\newtheorem{corollary}[theorem]{Corollary}

\theoremstyle{remark}

\crefname{theorem}{theorem}{theorems}
\Crefname{theorem}{Theorem}{Theorems}

\usepackage{hyperref}
\usepackage{cleveref}

\hypersetup{
  colorlinks=true,
  linkcolor=blue,
  citecolor=blue,
  urlcolor=blue
}

\begin{document}

\begin{center}
    \vspace*{2cm}
  
    {\LARGE\bfseries Complex numbers: uncommon knowledge\par}
    \vspace{1.5cm}
  
    {\LARGE Ivan Penkov\par}
    \vspace{2.5cm}
  
\end{center}

    \noindent
    {\small\textbf{Abstract.} In this note we unveil some subfields of the field of complex numbers which are well known to specialists, but seem to remain hidden for a wider mathematical audience. The existence and constructions of these fields follow directly from famous results of the 20-th century, however standard graduate textbooks avoid these topics.\par}
    \vspace{0.5cm}

\section{Foreword}
Every educated person knows, or knows about, complex numbers. However, asking someone to explain, or give a definition of complex numbers, may be intimidating and most people will say something like ``imaginary numbers". A more mathematically fluent person may come up with the more precise "numbers of the form $a + bi$ where $a$ and $b$ are real numbers and $i = \sqrt{-1}$". Asking to explain what a real number is a futile task unless you are talking to someone with a major in Mathematics. In that case, there a chance to hear that real numbers form a complete Archimedean totally ordered field, and less chance to hear that real numbers are nothing but Dedekind cuts of the rational numbers.
\par
Professional mathematicians have the background necessary for understanding complex numbers in depth. Nevertheless, our experience from talking to numerous colleagues (number theorists and field theorists excluded!) is that the existence of certain subfields of the field of complex numbers often remains hidden, although it follows directly from famous theorems in common textbooks. In this article we recall some exciting constructions within the field of complex numbers which can serve as illustrations to standard theorems and deserve to be more widely known. We hope you have fun!

\section{$\mathrm{Aut} \mathbb{R}$ versus $\mathrm{Aut}\mathbb{C}$}
In what follows, $i$ stands always for $ i = \sqrt{-1}$, i.e., a number such that $i^2 = -1$. Recall that a \textit{field} is a set with two operations, addition and multiplication, satisfying the usual rules of numbers, in particular $a(b+c) = ab + ac$. Subtraction and division are possible, except dividing by $0$.
An \textit{automorphism} of a field $F$ is a bijection $\sigma:F \to F$ such that
\begin{itemize}
    \item $\sigma(a+b) = \sigma(a) + \sigma(b)$
    \item $\sigma(ab) = \sigma(a)\sigma(b)$.
\end{itemize}
Automorphisms form a group, the group operation being composition of automorphisms.

Here is a basic fact which can trigger the discussion.

\begin{theorem}
    The automorphism group $\textnormal{Aut} \mathbb{R}$ of the field of real numbers $\mathbb{R}$ is trivial, i.e. consists of the identity automorphism only.
\end{theorem}

\begin{proof}
    Note that any automorphism of $\mathbb{R}$ preserves the set of non-negative numbers, and thus the order in $\mathbb{R}$. Indeed, non-negative numbers are preserved since they are exactly the elements of $\mathbb{R}$ which are squares of elements of $\mathbb{R}$.  Next: the rational numbers $\mathbb{Q}$ stay fixed under any automorphism of $\mathbb{R}$. Please provide a 1-line argument for this. Since the order is preserved, and since any irrational number is approximated arbitrarily closely by rational numbers which stay fixed by automorphisms, we conclude that irrational numbers are also fixed under any automorphism of $\mathbb{R}$.
\end{proof}

What can we say about the group of automorphisms of the field of complex numbers $\mathbb{C}$, which for the moment we think of as $\mathbb{R}\oplus\mathbb{R}i$? Any mathematician or physicist will tell that there is a non-trivial such automorphism, called \textit{complex conjugation}
$$ a + bi \mapsto a - bi.$$
This automorphism keeps the real numbers $a + 0i$ fixed, which agrees well with the above theorem. However, a key a question is: are there any other elements of the group $\mathrm{Aut} \mathbb{C}$ of automorphisms of $\mathbb{C}$? Most mathematicians know that the answer is positive, but some may have never asked themselves why. 

Let's start with a small challenge to the reader:

$$
\textrm{Is there any automorphism of the field } \mathbb{C} \textrm{ which sends } \sqrt{2} \textrm{ to} -\sqrt{2}?
$$
Hm...m. Depending on one's background, it could come as a surprise that the answer is YES. Here is how to construct such an automorphism of $\mathbb{C}$. One first considers the subfield $\{ a + b\sqrt{2} \} \subset \mathbb{C}$ where $a, b \in \mathbb{Q}$. This field has the obvious automorphism $\varphi_0 : \{ a + b\sqrt{2} \} \to \{ a + b\sqrt{2} \}$ which acts as $\varphi_0(a + b\sqrt{2}) = a - b\sqrt{2}$. Next, a standard theorem \cite[Ch.V, Thm.~3.8]{Hungerford} tells us that $\varphi_0$ can be extended to an automorphism $\varphi_1$ of the algebraic closure $\overline{\mathbb{Q}}$ of $\mathbb{Q}$. The field $\overline{\mathbb{Q}}$, i.e., the field of all algebraic numbers, is a subfield of $\mathbb{C}$ and a transcendence basis  $\{x_{\alpha}\}$ of $\mathbb{C}$ over $\overline{\mathbb{Q}}$ exists by Zorn's lemma. Recall that a transcendence basis is a maximal set of elements $\{x_{\alpha}\} \subset \mathbb{C}$ such that no polynomial expression involving finitely many elements $x_\alpha$, and having rational coefficients, vanishes except when the coefficients are all equal to zero. 
The minimal subfield of $\mathbb{C}$ containing $\overline{\mathbb{Q}}$ and $\{x_\alpha\}$ is denoted by $\overline{\mathbb{Q}}(\{x_{\alpha}\})$, and the automorphism $\varphi_1 : \overline{\mathbb{Q}} \to \overline{\mathbb{Q}}$ extends in an obvious way to  $\overline{\mathbb{Q}}(\{x_{\alpha}\})$ by leaving all $x_\alpha$ fixed.
Denoting this automorphism by $\varphi_2$, we use again \cite[Ch.V, Thm.~3.8]{Hungerford} to extend it to an automorphism $\varphi_3$ of the algebraic closure $\overline{\overline{\mathbb{Q}}(\{x_{\alpha}\})}$ of the field $\overline{\mathbb{Q}}(\{x_{\alpha}\})$. This latter field is nothing but $\mathbb{C}$. The so obtained automorphism $\varphi_3: \mathbb{C} \to \mathbb{C}$ sends $\sqrt{2}$ to $-\sqrt{2}$ by definition.
\par
OK, nothing very surprising so far. But where is the subfield $\mathbb{R}$ of $\mathbb{C}$ being sent to by $\varphi_3$? Clearly, $\varphi_3(\mathbb{R}) \neq \mathbb{R}$ since $\mathrm{Aut} \mathbb{R} = \{ \mathrm{id} \}$ and $\varphi_3(\sqrt{2}) = -\sqrt{2}$.
So what is the subfield of $\varphi_3(\mathbb{R}) \subset \mathbb{C}$? Of course it is isomorphic to $\mathbb{R}$, but where is it in the usual Cartesian picture $\mathbb{C} = \mathbb{R} \oplus \mathbb{R}i = \mathbb{R}^2$? Well, it is spread all over $\mathbb{R}^2$, everwhere dense in the usual topology of $\mathbb{R}^2$. Why? Because $\varphi_3(\mathbb{R})$ necessarily contains all vectors of the form $a + bx$ where $a$ and $b$ are rational numbers and $x$ is a nonzero complex number (vector in $\mathbb{R}^2$) such that $x \in \varphi_3(\mathbb{R}) \backslash \mathbb{R}$. Interesting, right?

\section{Definition of $\mathbb{C}$, and copies of $\mathbb{R}$ within $\mathbb{C}$}

In order to go deeper into the picture, we need to lay our cards on the table. Defining $\mathbb{C}$ as the field $\{a + bi\}$ for $a,b \in \mathbb{R}$ is perfectly valid, but automatically plants a copy of the real numbers inside of $\mathbb{C}$. If one does not think further, one may naively perceive this copy as canonical, or natural if you prefer. However, the following iconic theorem yields a definition of the field of complex numbers without any reference to real numbers. 

\begin{theorem} \label{Steinitz}
    (Steinitz \textnormal{\cite{Steinitz}}, see also \textnormal{\cite[Ch.VI, Thm.~1.12]{Hungerford}}) Any algebraically closed field of characteristic 0 and of cardinality $2^{\aleph_0}$ is isomorphic to the field of complex numbers. 
\end{theorem}

Characteristic 0 means that the element $1 + 1 + \cdots +1 $ never equals 0 (this can happen for instance in finite fields) and $2^{\aleph_0}$ stands for the cardinality of the set of all subsets of a countable set.
\par
If we think of $\mathbb{C}$ as an algebraically closed field satisfying the two conditions in Theorem 3.1, then the question of how many copies of the field $\mathbb{R}$ are "hiding" within $\mathbb{C}$ looks totally different. And since, as we already demonstrated, the group $\mathrm{Aut} \mathbb{C}$ does not fix any such copy, there is no reason to expect that some copy of $\mathbb{R}$ within $\mathbb{C}$ is more canonical than another.
\par
In general, a field extension $F \subset \mathbb{C}$ has two essential invariants, namely the degree of $\mathbb{C}$ over $F$ and the transcendence degree $\mathrm{trdeg}_F \mathbb{C} $. By definition, the \textit{degree} of $\mathbb{C}$ over $F$, $[\mathbb{C}:F]$, is just the dimension of $\mathbb{C}$ as a vector space over $F$, and the \textit{transcendence degree} $\mathrm{trdeg}_F \mathbb{C}$ is the cardinality of a transcendence basis of $\mathbb{C}$ over $F$. If $\mathrm{trdeg}_F \mathbb{C} \neq 0$ then $[ \mathbb{C}: F ]$ is necessarily infinite. Moreover, both $[ \mathbb{C}: F ]$ and
$\mathrm{trdeg}_F \mathbb{C}$ can be finite, countable, or uncountable (at most equal to $2^{\aleph_0}$). The simplest case is when $[ \mathbb{C}: F ]$ is finite and consequently $\mathrm{trdeg}_F \mathbb{C} = 0$. This case is covered by another iconic theorem.

\begin{theorem}
    (Artin-Schreier \textnormal{\cite{ArtinSchreier}}, see also \textnormal{\cite[Ch.XI]{Lang},\cite{ConradArtinSchreier}}). Let $F\subset\mathbb{C}$ be a finite field extension i.e., $[\mathbb{C}:F] < \infty$. Then $[\mathbb{C}:F] = 2$, and consequently $\mathbb{C} = F \oplus Fi = \{a + bi\} \textrm{ for } a,b \in F $.
\end{theorem}

Of course, the subfield $\mathbb{R}$ of the standard picture of $\mathbb{C}$ as $\mathbb{R} \oplus \mathbb{R}i$ satisfies the condition of the Artin-Schreier Theorem. But is a general subfield $F\subset \mathbb{C}$ satisfying this condition necessarily isomorphic to $\mathbb{R}$?
In general this is not true! Did you know that?

In order to shed more light on the subfields $F\subset\mathbb{C}$ as in Theorem (3.2), we need to discuss

\section{Real closed fields}

A field $F$ is \textit{real closed} if it is not algebraically closed but $F(i) = F \oplus Fi$ is algebraically closed. See \cite{RealClosedFieldWiki}   for 8 further equivalent definitions of a real closed field, and check also \cite[Ch.XI]{Lang}. An important property of a real closed field $F$ is that it has a canonical linear order compatible with addition and multiplication in  the usual way, and such that the non-negative elements are precisely the squares in $F$. Note that the characteristic of any ordered field is 0 since otherwise 

$$1+1+ \cdots +1 = -1 < 0$$
for some sum $1+1+ \cdots +1$. In addition, an ordered field $F$ is \textit{Archimedean} if for any $x\in F$ there exists a positive integer $n$ such that $x \leqslant n$. 

Of course, the field $\mathbb{R}$ is real closed but it is by far not the only real closed field. Well-known examples of real closed fields are the field $\mathcal{P}$ of Puiseux series and the field $\mathcal{R}$ of Levi-Civita series, see for instance \cite[Prop.~II.8]{Serre} and \cite{Ribenboim1992}. Puiseux series are formal power series $ \sum_{j\geqslant1}a_j t^{k_j}$ where ${k_j}$ is an increasing sequence of rational numbers with bounded denominators, and $a_j \in \mathbb{R}$. For instance, $ t^{-\frac{1}{12}} + \sum_{j \geqslant 1} t^{\frac{2j+1}{5}}$ is a Puiseux series. The Levi-Civita field is larger than the field of Puiseux series and consists of formal power series of the same form as Puiseux series but without the condition that the denominators of the sequence of rational numbers ${{k_j}}$ are bounded.
\par
The fields $\mathcal{P}$ and $\mathcal{R}$ have characteristic 0 and it is easy to show that they both have cardinality equal to $2^{\aleph_0}$. By Steinitz's Theorem, their algebraic closures $\overline{\mathcal{P}} = \mathcal{P}(i)$ and $\overline{\mathcal{R}} = \mathcal{R}(i)$ are both isomorphic to $\mathbb{C}$. So, for instance we have this picture

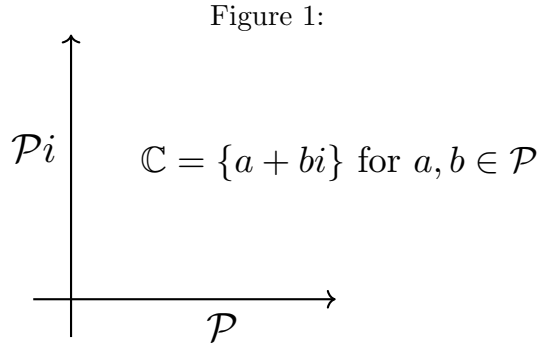
\begin{figure}[h!]
    \centering
    \caption{\label{diag:P_plot}}
    \begin{tikzpicture} 
        \draw[thick,->] (-0.5,0) -- (3.5,0);
            \node[scale=1.4] [left] at (0,2) {$\mathcal{P}i$};
        \node[scale=1.26]  at (3.55,1.8) {$ \mathbb{C} = \{a + bi\}$ for $a, b \in \mathcal{P}$};
        \draw[thick,->] (0,-0.5) -- (0,3.5);
            \node[scale=1.4] [below] at (2,0) {$\mathcal{P}$};
    \end{tikzpicture}

\end{figure}
\noindent
In other words, $\mathbb{C}$ is represented not by the plane $\mathbb{R}^2$ but the plane $\mathcal{P}^2$ where $\mathcal{P}$ is the field of Puiseux series which has $\mathbb{R}$ as proper subfield.
\par
It gets even more exciting. The representation of $\mathbb{C}$ as  $\mathcal{P} \oplus \mathcal{P}i$ yields immediately an automorphism of $\mathbb{C}$

$$ a + bi \mapsto a - bi \textrm{ for } a, b \in \mathcal{P},$$
and this automorphism is not usual complex conjugation since the fixed point field is $\mathcal{P}$ and not $\mathbb{R}$. Though, wait a minute: is it clear that $\mathcal{P}$ is not isomorphic
to $\mathbb{R}$? Yes, this is clear because $\mathcal{P}$ contains a proper subfield isomorphic to $\mathbb{R}$, and $\mathbb{R}$ does not have such a subfield (prove this!).

In addition, as we claimed before, the field $\mathcal{P}$ is real closed hence has a  canonical linear order. Let's take a look at this order. Let $n$ be a positive integer. Note first that the element $y = 1  + nt + n^2 t^2 + \cdots + n^kt^k+ \cdots$ is a square in $\mathcal{P}$ and is thus positive. One can prove this by computing a square root of $y$ explicitly by recursion: 
$ \sqrt{y} = 1 + \frac{n}{2}t + \frac{3}{8}n^2 t^2 + \cdots $. Therefore, 

$$
\frac{1}{1-nt}= y > 0 \implies 1 - nt > 0 \implies t < \frac{1}{n} \:\: \forall n.
$$
On the other hand, $t > 0$ since $t$ is also a square in $\mathcal{P}$ ($t = (t^{\frac{1}{2}})^2$), and we obtain

$$ 0 < t < \frac{1}{n} \qquad \forall n.$$
This means that the field $\mathcal{P}$ is non-Archimedean (!), which is of course well known.
\par
We arrive to the following conclusion: if $F$ is any real closed field of cardinality $2^{\aleph_0}$, then its algebraic closure $\overline{F} = F \oplus Fi$ is isomorphic to $\mathbb{C}$, and we have \Cref{diag:P_plot} with $\mathcal{P}$ replaced by $F$.
\par

Before giving the next example, we recall that any ordered field $F$ has a real closure $\vec{F}$. By definition $\vec{F}$ is a minimal real closed field containing $F$. The proof of existence for $\vec{F}$ is easy: define $\vec{F}$ as a maximal ordered intermediate field $F \subset \vec{F} \subset \overline{F}$ such that the linear order on $\vec{F}$ extends that of $F$. Zorn's Lemma implies that $\vec{F}$ exists, and similarly to algebraic closure, real closure is not unique but is unique up to isomorphism.

We can now give an example of a real closed field $F$ which is a proper subfield of $\mathbb{R}$ and still $\mathbb{C} \cong F \oplus F i$. In this way, the canonical ordering on $F$ will be Archimedean but $F$ will not be isomorphic to $\mathbb{R}$. Fix a transcendence basis $\{x_a\}$ of $\mathbb{R}$ over $\mathbb{Q}$ so that the transcendental number $\pi$ equals $x_{\alpha_0}$ for some $\alpha_0$. Consider the linearly ordered field $\mathbb{Q}(\{x_\beta\})$ where $\beta$ runs over all indices of the basis $\{x_\alpha\}$ except $\alpha_0$. Denote by $\mathbb{R}_{\backslash\{\pi\}}$ a real closure of $\mathbb{Q}(\{x_\beta\})$. Then the algebraic closure $\overline{\mathbb{R}_{\backslash\{\pi\}}}$ of $\mathbb{R}_{\backslash\{\pi\}}$ is isomorphic to $\mathbb{C}$ and $\mathbb{C} \simeq \mathbb{R}_{\backslash\{\pi\}} \oplus \mathbb{R}_{\backslash\{\pi\}}i$.
We now have the picture

\begin{figure}[h!]
    \centering
    \caption{\label{diag2}}
    \begin{tikzpicture}
        \draw[thick,->] (-0.5,0) -- (3.5,0);
            \node[scale=1.3] [left] at (0,2) {$\mathbb{R}_{\backslash\{\pi\}}i$};
        \node[scale=1.2]  at (3.85,1.5) {$\mathbb{C} = \{a + bi \}$ for $a, b \in \mathbb{R}_{\backslash\{\pi\}}$};
        \draw[thick,->] (0,-0.5) -- (0,3.5);
            \node[scale=1.3] [below] at (2,0) {$\mathbb{R}_{\backslash\{\pi\}}$};
    
    \end{tikzpicture}
\end{figure}
\noindent
where $\mathbb{C}$ is represented as the cartesian product $(\mathbb{R}_{\backslash\{\pi\}})^2$ for a subfield $\mathbb{R}_{\backslash\{\pi\}}$ of $\mathbb{R}$ in which the transcendental number $\pi$ is missing. Hence, for every nonzero field homomorphism $ \psi : \mathbb{R} \to \mathbb{C}$, the number $\psi(\pi)$ will not be on the $\mathbb{R}_{\backslash\{\pi\}}$ - axis and will have the form $a + bi$ for $a,b \in \mathbb{R}_{\backslash\{\pi\}}, b \neq 0$!
\newline
\section{Aut$\mathbb{C}$ revisited}

Let's come back to the question of how many copies of the field $\mathbb{R}$ are there in $\mathbb{C}$. This is closely related to the structure of the group Aut$\mathbb{C}$ since it moves around the subfields of $\mathbb{C}$, in particular those isomorphic to $\mathbb{R}$.

\begin{theorem}
    The group $\textnormal{Aut} \mathbb{C}$ is "huge": it has cardinality $2^{2^{\aleph_0}}$.
\end{theorem}

\begin{proof}
    Fix a transcendence base $\{x_\alpha\}$ of $\mathbb{C}$ over $\mathbb{Q}$. Then every element of the permutation group of the set $\{x_\alpha\}$ defines an automorphism of the field $\mathbb{Q}(\{x_\alpha\})$. Any such automorphism extends to an automorphism of $\overline{\mathbb{Q}(\{x_\alpha\})} = \mathbb{C}$ by \cite[Ch.V, Thm.3.8]{Hungerford}. 
    \par
    Since the permutation group of a set with cardinality $2^{\aleph_0}$ has cardinality $2^{2^{\aleph_0}}$, we see that the group Aut$\mathbb{C}$ has cardinality at least $2^{2^{\aleph_0}}$. On the other hand, the set of all maps from $\mathbb{C}$ to $\mathbb{C}$ has cardinality $2^{2^{\aleph_0}}$, so the claim follows.
\end{proof}

Next we have 

\begin{theorem} \label{realclosedThm}
    Let $F$ and $F'$ be real closed subfields of $\mathbb{C}$ such that $F' \simeq F$ and  $\textnormal{trdeg}_{F'}\: \mathbb{C} = \textnormal{trdeg}_{F}\: \mathbb{C}$. Then $\sigma(F) = F'$ for some $\sigma \in $ \textnormal{Aut}$\,\mathbb{C}$.
\end{theorem}

\begin{proof}
     Choose a transcendence base $\{x_{\alpha}\}$ of $F$ over $\mathbb{Q}$ and a transcendence base $\{x_{\overline{\alpha}}\}$ of $\mathbb{C}$ over $F$. Fix an isomorphism $F \overset{\sim}{\to} F'$ (this isomorphism is already order preserving!) and use it to transfer the base $\{x_{\alpha}\}$ to $F'$. Let this transferred base be $\{x'_{\alpha}\}$, and let $\{x'_{\overline{\alpha}}\}$ be a transcendence base of $\mathbb{C}$ over $F'$ . We use here the condition $\textrm{trdeg}_F \mathbb{C} = \textrm{trdeg}_{F'} \mathbb{C}$ in order to be able to parametrize the bases $\{x_{\overline{\alpha}}\}$ and $\{x'_{\overline{\alpha}}\}$ by the same index set.
    \par
    Consider now the isomorphism
    $$
        \mathbb{Q}(x_\alpha, x_{\overline{\alpha}}) \mapsto \mathbb{Q}(x'_\alpha, x'_{\overline{\alpha}})
    $$
    sending $x_\alpha$ to $x'_\alpha$ and $x_{\overline{\alpha}}$ to $ x'_{\overline{\alpha}}$, and extend it by Theorem \cite[Ch.V, Thm.~3.8]{Hungerford} to an isomorphism of the common algebraic closure $\mathbb{C}$ of $\mathbb{Q}(x_\alpha, x_{\overline{\alpha}})$ and $\mathbb{Q}(x'_\alpha, x_{\overline{\alpha}})$. This automorphism carries $F$ to $F'$ since $F$ (respectively, $F'$) is the real closure of the field $\mathbb{Q}(x_\alpha)$ (respectively, of $\mathbb{Q}(x'_\alpha))$ in $\mathbb{C}$.
\end{proof}

\begin{corollary}
        There are $2^{2^{\aleph_0}}$ different subfields $F\subset \mathbb{C}$ such that $F$ is isomorphic to $\mathbb{R}$ and $[\mathbb{C}: F ] = 2$.
\end{corollary}
\begin{proof}
        Represent $\mathbb{C}$ as $\mathbb{R}^2 = \mathbb{R} \oplus \mathbb{R}i$. The only automorphism of $\mathbb{C}$ which preserves the fixed subfield $\mathbb{R}$ is the complex conjugation $\overline{(\cdot)}$. The image of $\mathbb{R}$ under any other non-trivial automorphism of $\mathbb{C}$ is a subfield $F$ isomorphic to $\mathbb{R}$ but not coinciding with $\mathbb{R}$. Moreover, the images of $\mathbb{R}$ under any two such automorphisms $\sigma_1 \neq \sigma_2$ coincide if and only if $\sigma_1 = \sigma_2 \circ \overline{(\:\:\:)}.$ Otherwise $\sigma_1 \circ \sigma^{-1}_2$ would induce a non-trivial automorphism of $\mathbb{R}$. Since there are $2^{2^{\aleph_0}}$ elements of the group Aut$\mathbb{C}$ no two of which are related by the equality $\sigma_1 = \sigma_2 \circ \overline{(\:\:\:)}$,  the statement follows.
    \end{proof}

    The above corollary allows one to finally part with the illusion (if one ever had it) that there is a distinct copy of $\mathbb{R}$ in $\mathbb{C}$.

\section{Defining complex conjugation}

Here is an invariant definition of complex conjugation: this is an element $\sigma \in$ Aut$\mathbb{C}$ satisfying $\sigma^2 = id$ (in other words, an \textit{involution} $\sigma \in$ Aut$\mathbb{C}$) and the additional condition that the subfield of fixed points of $\sigma$ is isomorphic to $\mathbb{R}$.
\par
Such a definition accounts for all copies of $\mathbb{R}$ in $\mathbb{C}$ with $\textrm{trdeg}_\mathbb{R} \mathbb{C} = 0$. The fixed points of a general involution $\sigma \in$ Aut$\mathbb{C}$ can be any real closed subfield $F \subset \mathbb{C}$ with $[\mathbb{C} : F] =2,$ and $F$ can be Archimedean as well as non-Archimedean.

\section{One more surprising connection}

Let us now ask what is the relation of the field of complex numbers $\mathbb{C}$ with A. Robinson's field of non-standard real numbers $^*\mathbb{R}$? Most mathematicians are aware that such a field exists but have not taken a deeper look at its definition. The term "non-standard analysis" is completely established by now, and the emphasis is on "analysis", while the definition of the field $^*\mathbb{R}$ is totally within the framework of standard algebra textbooks. Regrettably, it is missing in them.
\par
Here is a brief outline of the definition of $^*\mathbb{R}$. Consider the commutative ring $\mathcal{F}_{\mathbb{R}}$ of functions $f:\mathbb{Z}_{>0} \to \mathbb{R}$ with pointwise multiplication and addition. Let us define a maximal ideal in $\mathcal{F}_\mathbb{R}$ by imposing conditions on the nullset of functions in this ideal. Fix a proper subset $\mathcal{U}$ of the power set of $\mathbb{Z}_{>0}$ such that:
\begin{itemize}
    \item $\mathcal{U}$ contains all subsets of $\mathbb{Z}_{>0}$ with finite complements,
    \item if $A \in \mathcal{U}$ and $A' \in \mathcal{U}$ , then $A\cap A' \in \mathcal{U}$,
    \item if $A \in \mathcal{U}$ and $A'\supset  A, $ then $ A' \in \mathcal{U}$,
    \item for any $A \subset \mathbb{Z}_{>0}$, exactly one of the sets $A$ and $\mathbb{Z}_{>0}\backslash  A$ belongs to $\mathcal{U}$.
\end{itemize}
The existence of $\mathcal{U}$ follows from Zorn's Lemma  (prove!), and $\mathcal{U}$ is known as an \textit{ultrafilter extending Fréchet's filter}, see \cite{Garcia}, \cite{Stroyan-Luxemburg}.
\par
Then one checks immediately that the functions $f \in \mathcal{F}_{\mathbb{R}}$ whose null-set belongs to $\mathcal{U}$ form a maximal ideal $\mathcal{I}_\mathcal{U}$ of $\mathcal{F}_{\mathbb{R}}$. Consequently, the quotient ring $\mathcal{F}_{\mathbb{R}} / \mathcal{I}_\mathcal{U}$ is a field. The fact that, up to isomorphism, the field  $\mathcal{F}_{\mathbb{R}} / \mathcal{I}_\mathcal{U}$ does not depend on the set $\mathcal{U}$ follows from \cite{egh}, and the proof relies on the continuum hypothesis. By definition, $^*\mathbb{R}$ is the field $\mathcal{F}_{\mathbb{R}} / \mathcal{I}_\mathcal{U}$.
\par
The field $^*\mathbb{R}$ is ordered. Indeed, let $f,g \in \mathcal{F}_{\mathbb{R}}$. Then there is a subset $A \subset \mathbb{Z}_{\geqslant 0}$ such that $f(j) \geqslant g(j)$ for $j \in A$ and $f(j) \leqslant g(j)$ for $j \notin A$. One of the subsets $A$ and $\mathbb{Z}_{>0} \backslash A$ belongs to $\mathcal{U}$, and we denote this set by $B$. The inequality on $\mathbb{Z}_{>0} \backslash B$ determines the inequality for the images $\overline{f}, \overline{g} \in   {^*\mathbb{R}}$:

$$
\overline{f} \leqslant \overline{g} \iff f(j) \leqslant g(j) \textrm{ for } \mathbb{Z}_{>0} \backslash B.
$$

And finally, $^*\mathbb{R}$ is real closed and non-Archimedean.
The fact that $^*\mathbb{R}$ is real closed is proved in \cite{Robinson}.
The fact that $^*\mathbb{R}$ is non-Archimedean is obvious. Indeed, the image $\bar{f} \in {^*\mathbb{R}}$ of a function $f \in \mathcal{F}_\mathbb{R}$ with positive values tending to zero for $ j \to \infty$ is a positive infinitesimal in $^*\mathbb{R}$, i.e., $\lim\limits_{j \to \infty} f(j) = 0, f(j) > 0 \quad \forall j \in \mathbb{Z}_{>0}$ implies

$$
0 < \overline{f} < \dfrac{1}{n} \quad \forall n \in \mathbb{Z}_{>0}.
$$

What is the cardinality of $^*\mathbb{R}?$ It clearly is $2^{\aleph_0}$ as $^*\mathbb{R}$ contains $\mathbb{R}$, and the cardinality of $\mathcal{F}_{\mathbb{R}}$ is also $2^{\aleph_0}$. Hence the algebraic closure $\overline{^*\mathbb{R}}$ of $^*\mathbb{R}$ is isomorphic to $\mathbb{C}$, and we obtain the picture 

\begin{figure} [h!]
    \centering
    \caption{\label{diag2}}
    \begin{tikzpicture}
        \draw[thick,->] (-1,0) -- (4.5,0);
            \node[scale=1.3] [left] at (0,2) {$^*\mathbb{R}i$};
        \node[scale=1.2]  at (3.85,1.5) {$\mathbb{C} = \{a + bi \}$ for $a, b \in {^*\mathbb{R}}
        $};
        \fill (-1,0) circle (0.05cm);
        \fill (-0.75,0) circle (0.05cm);
        \fill (-0.5,0) circle (0.05cm);
        \fill (-0.25,0) circle (0.05cm);
        \fill (0,0) circle (0.05cm);
        \fill (0.25,0) circle (0.05cm);
        \fill (0.5,0) circle (0.05cm);
        \fill (0.75,0) circle (0.05cm);
        \fill (1,0) circle (0.05cm);
        \fill (1.25,0) circle (0.05cm);
        \fill (1.5,0) circle (0.05cm);
        \fill (1.75,0) circle (0.05cm);
        \fill (2,0) circle (0.05cm);
        \fill (2.25,0) circle (0.05cm);
        \fill (2.5,0) circle (0.05cm);
        \fill (2.75,0) circle (0.05cm);
        \fill (3,0) circle (0.05cm);
        \fill (3.25,0) circle (0.05cm);
        \fill (3.5,0) circle (0.05cm);
        \fill (3.75,0) circle (0.05cm);
        \fill (4,0) circle (0.05cm);
        \fill (4.25,0) circle (0.05cm);
        \draw[thick,->] (0,-0.5) -- (0,3.5);
            \node[scale=1.3] [below] at (2,0) {$^*\mathbb{R}$};
    \end{tikzpicture},
\end{figure}
\noindent
where the thickened dots of the horizontal axis are the elements of $^*\mathbb{R}$ which belong to $\mathbb{R}$. Had you ever imagined such a picture?

\section{Short word on topologies of $\mathbb{C}$}

For most colleagues doing analysis, the field of complex numbers $\mathbb{C} = \mathbb{R} \oplus \mathbb{R}i \simeq \mathbb{R}^2$ automatically comes with the standard topology from $\mathbb{R}^2$. If one thinks in this way, $\mathbb{C}$ is no longer just a field but a Banach algebra, the norm homomorphism $$\mathbb{C} \to \mathbb{R}_{\geqslant0}, \quad a+bi \mapsto \sqrt{a^2 + b^2}$$ being a part of the definition of $\mathbb{C}$. Moreover, the only continuous automorphism of this Banach algebra is the usual complex conjugation 

$$ a + bi \mapsto a - bi  .$$

However, one of our points above is that the field $\mathbb{C}$ has no canonical structure of a Banach algebra, since for this one needs a copy of $\mathbb{R}$ in $\mathbb{C}$. All copies of $\mathbb{R}$ in $\mathbb{C}$ with $[\mathbb{C}: \mathbb{R}] = 2$ are equally good. Moreover, they define different Banach algebra structures on $\mathbb{C}$ (prove this!), and consequently different topologies on $\mathbb{C}$. 
The fact that any two such copies of $\mathbb{R}$ are conjugate by an automorphism of $\mathbb{C}$ (according to Theorem \ref{realclosedThm}) implies that all these Banach algebra structures, and topologies on $\mathbb{C}$, are isomorphic.
\par
One can extend this thought by defining general norm homomorphisms  $$\mathbb{C} \to F_{\geqslant0}, \quad a + bi \mapsto \sqrt{a^2 + b^2}$$ where $F \subset \mathbb{C}$ is now any given (real closed) subfield with $[\mathbb{C}: F] = 2.$ In particular, there are non-standard norms

$$ 
\mathbb{C} \to {^*\mathbb{R}_{\geqslant0}}
$$
and they induce respective topologies on the field $\mathbb{C}$. With respect to such non-standard norms, there exist infinitesimal complex numbers! The interested reader can think about this further.

\section{Conclusion}

Things look very simple when one builds the field of complex numbers $\mathbb{C}$ from "bottom to top", starting from rational numbers, passing to real numbers (this is the most nontrivial step) and then taking algebraic closure. However, when one learns Steinitz's Theorem which characterizes the field $\mathbb{C}$ without any reference to real numbers, one can study this field from "top to bottom".

In particular, it is interesting to think on how many and what other fields embed into $\mathbb{C}$. The shocking answer is that any field of characteristic $0$ and of cardinality less or equal to $2^{\aleph_0}$ can be realized as a subfield of $\mathbb{C}$. In the usual picture of $\mathbb{C}$ as $\mathbb{R} \oplus \mathbb{R}i \simeq \mathbb{R}^2$ most of these other fields are hidden, and we enjoyed unveiling a few of them. Hopefully the reader enjoyed that too! 

\section*{Aknowledgment}
This note was inspired by discussions with my friend Todor Todorov on the "hidden" subfields of $\mathbb{C}$ arising through the Steinitz Theorem.

    \vspace{0.5cm}

    Ivan Penkov 
    
    Constructor University, 28759 Bremen, Germany

    Email adress: ivanpenkov@yahoo.com
\end{document}